\documentclass[10pt,a4paper,british]{amsart}
\usepackage{mathptmx}
\usepackage[scaled=0.9]{helvet}
\usepackage{courier}
\usepackage[T1]{fontenc}
\usepackage[utf8]{inputenc}
\usepackage{xcolor}
\usepackage{pdfcolmk}
\usepackage{float}
\usepackage{mathtools}
\usepackage{amstext}
\usepackage{amsthm}
\usepackage{amssymb}
\usepackage{graphicx}
\usepackage[all]{xy}
\PassOptionsToPackage{normalem}{ulem}
\usepackage{ulem}

\makeatletter

\pdfpageheight\paperheight
\pdfpagewidth\paperwidth

\providecolor{lyxadded}{rgb}{0,0,1}
\providecolor{lyxdeleted}{rgb}{1,0,0}

\numberwithin{equation}{section}
\numberwithin{figure}{section}
\theoremstyle{plain}
\newtheorem{thm}{\protect\theoremname}[section]
  \theoremstyle{plain}
  \newtheorem{fact}[thm]{\protect\factname}
  \theoremstyle{plain}
  \newtheorem{prop}[thm]{\protect\propositionname}
  \theoremstyle{plain}
  \newtheorem{cor}[thm]{\protect\corollaryname}
  \theoremstyle{plain}
  \newtheorem{lem}[thm]{\protect\lemmaname}

\usepackage{amssymb,amsmath}
\usepackage{microtype}
\usepackage{bbm}
\usepackage{color}

\newcommand{\e}{\mathrm e}
\DeclareMathOperator{\Int}{Int}

\DeclareMathOperator{\card}{card}

\renewcommand{\hat}{\widehat}
\renewcommand{\phi}{\varphi}
\renewcommand{\Im}{\mathfrak{Im}}

\renewcommand{\emptyset}{\varnothing}

\renewcommand{\lg}{\log}
\newcommand{\G}{\mathcal{G}}
\newcommand{\X}{{\mathbb{H}}}

\DeclareMathOperator{\id}{id}

\usepackage{pgf,tikz}
\usetikzlibrary{arrows}

\definecolor{qqqqff}{rgb}{0.0,0.0,1.0}

\def\card {\mathop {\hbox{\rm card}}}

\def\R{\mathbb{R}}
\def\H{\mathbb{H}}

\def\N{\mathbb{N}}
\def\Z{\mathbb{Z}}

\SetSymbolFont{operators}{bold}{OT1}{cmr}{bx}{n}
\SetSymbolFont{letters}{bold}{OML}{cmm}{b}{it}
\SetSymbolFont{symbols}{bold}{OMS}{cmsy}{b}{n}

\usepackage{babel}

\numberwithin{equation}{section} 
\numberwithin{figure}{section} 

\makeatother

\usepackage{babel}
  \providecommand{\corollaryname}{Corollary}
  \providecommand{\factname}{Fact}
  \providecommand{\lemmaname}{Lemma}
  \providecommand{\propositionname}{Proposition}
\providecommand{\theoremname}{Theorem}

\begin{document}

\title{A Multifractal Analysis for Cuspidal Windings on Hyperbolic Surfaces}

\author{Johannes Jaerisch}
\address{Department of Mathematics, Faculty of Science and Engineering, Shimane
University, Nishikawatsu 1060, Matsue, Shimane, 690-8504 Japan.}

\email{jaerisch@riko.shimane-u.ac.jp}

\urladdr{www.math.shimane-u.ac.jp/\textasciitilde{}jaerisch/}

\author{Marc Kesseböhmer}

\address{Fachbereich 3 – Mathematik und Informatik, Universität Bremen, Bibliothekstr.
1, 28359 Bremen, Germany}

\email{mhk@math.uni-bremen.de}

\urladdr{www.math.uni-bremen.de/stochdyn}

\author{Sara Munday}

\address{Dipartimento di Matematica, Università di Pisa, Largo Pontecorvo 5,
Pisa, Italy}

\email{sara.munday@dm.unipi.it}

\date{\today}

\thanks{JJ was supported by the JSPS KAKENHI 90741869 and by the Research
Fellowship of the German Research Foundation (DFG-grant Ja 2145/1-1).
MK was supported by the German Research Foundation (DFG) grant \emph{Renewal
Theory and Statistics of Rare Events in Infinite Ergodic Theory} (Geschäftszeichen
Ke 1440/2-1). JJ and MK were supported by the DFG Scientific Network
\emph{Skew product dynamics and multifractal analysis} of the German
Research Foundation (DFG-grant Oe 538/3-1).}
\begin{abstract}
In this paper we investigate the multifractal decomposition of the
limit set of a finitely generated, free Fuchsian group with respect to the
mean cusp winding number. We will completely determine its multifractal
spectrum by means of a certain free energy function and show that
the Hausdorff dimension of sets consisting of limit points with the
same scaling exponent coincides with the Legendre transform of this
free energy function. As a by-product we generalise previously obtained
results on the multifractal formalism for infinite iterated function systems to the setting of infinite
graph directed Markov systems.

\end{abstract}

\maketitle

\section{Introduction and statement of results}

In this paper we carry out a multifractal analysis of cusp windings of the geodesic
flow on $\X/G$, for a finitely generated,  free, non-elementary
Fuchsian group $G$ with parabolic elements acting on the upper half-space model $(\X,d)$
of $2$-dimensional hyperbolic space. Recall that to each $x$ in
the radial limit set $L_{r}(G)$ of $G$ one can associate an infinite
word whose letters come from the symmetric set $G_{0}$ of generators of $G$.
Note that $G$ can be
written as a free product $G=H\ast\Gamma$, where $H=\langle h_{1}\rangle\ast\ldots\ast\langle h_{u} \rangle$
denotes the free product of finitely many elementary hyperbolic groups,
and $\Gamma=\langle \gamma_{1} \rangle\ast\ldots\ast\langle \gamma_{v} \rangle$ denotes the free
product of finitely many parabolic subgroups of $G$ such that $ \langle\gamma_{i}\rangle$
is the parabolic subgroup of $G$ associated with the parabolic fixed
point $p_{i}$ (see also  \cite{MR2041265}). We will always assume that $v\geq1$; furthermore, the fact that $G$ is non-elementary
implies that $u+v>1$. Clearly, $ \langle\gamma_{i}\rangle\cong\mathbb{Z}$ and $\gamma_{i}^{\pm 1}(p_{i})=p_{i}$,
for all $i=1,\ldots,v$. It is well known \cite{MR1484767} that the Poincaré exponent
$\delta=\delta_{G}$ of $G$ coincides with the Hausdorff dimension $\dim_{\mathrm{H}}\left(L_{r}\left(G\right)\right)$
of the radial limit set of $G$.

There is a natural coding of the limit set by infinite sequences over
the set of generators. That is, with $F$ referring to the Dirichlet
fundamental domain of $G$ containing $i\in\X$, the images of $F$ under
$G$ tessellate $\X$ and each side of each of the tiles is uniquely
labelled by an element of $G_{0}$. The hyperbolic ray $s_{x}$ from
$i$ towards $x\in L_{r}(G)$ must traverse infinitely many of
these tiles, and the infinite word expansion associated with $x$
is then obtained by progressively recording, starting at $i$, the
generators of the sides at which $s_{x}$ exits the tiles. In this
way we derive an infinite  word $\omega=\left(\omega_{1},\omega_{2},\dots\right)\in G_{0}^{\N}$,
which is necessarily reduced,
where reduced means that $\omega_{n}\omega_{n+1}\neq\id$, for all $n\in\N$.
We then form a sequence of blocks $(B_{n})_{n\in\N}$ in this word in the following way. Each
hyperbolic generator that appears in the word is called a block of length $1$. Further,
if the same parabolic generator appears consecutively exactly
$n$ times, then this is called a block of length $n$. By construction,
such a block of length $n$ corresponds to the event that the projection
of $s_{x}$ onto $\X/G$ spirals precisely $n-1$ times around a
cusp of $\X/G$. This motivates our definition of the \emph{cusp winding
process} $(a_{k})$ by setting, for each $x\in L_r(G)$,
\[
a_{k}:=a_k(x):=  |B_k|-1,
\]
where $B_k:=B_k(x)$ denotes the $k$-th block in the infinite word associated to $x$ and $|B_k|$ denotes its length.
Our main aim is to investigate the fluctuation of a certain asymptotic exponential
scaling associated to this process, thereby extending results in \cite{MR2672614,MR2719683} for $\mbox{PSL}_{2}\left(\mathbb{Z}\right)$.
We say that the \emph{mean cusp-winding number} of $x\in L_{r}\left(G\right)$ is
given by
\[
\lim_{n\rightarrow\infty}\frac{2\sum_{k=1}^{n}\log^{+}(a_k)}{d\left(B_{1}\cdots B_{n}\left(i\right),i\right)},
\]
whenever the limit exists. Here, $\log^{+}(a)\coloneqq\left\{
                                                        \begin{array}{ll}
                                                          \log(a), & \hbox{when $a>0$;} \\
                                                          0, & \hbox{when $a=0$.}
                                                        \end{array}
                                                      \right.
$.
The fluctuation of this quantity is captured in the following level sets with prescribed scaling constant $\alpha\in\R$ given by
\[
\mathcal{F}_{\alpha}\coloneqq\left\{ x\in L_{r}\left(G\right)\mid\lim_{n\rightarrow\infty}\frac{2\sum_{k=1}^{n}\log^{+}(a_k)}{d\left(B_{1}\cdots B_{n}\left(i\right),i\right)}=\alpha\right\},
\]
and
\[
\mathcal{F}_{\alpha}^{*}\coloneqq \begin{cases}
\left\{ x\in L_{r}\left(G\right)\mid\limsup_{n\rightarrow\infty}\frac{2\sum_{k=1}^{n}\log^{+}(a_k)}{d\left(B_{1}\cdots B_{n}\left(i\right),i\right)}\geq\alpha\right\} , & \alpha\geq\alpha_{0},\\
\left\{ x\in L_{r}\left(G\right)\mid\liminf_{n\rightarrow\infty}\frac{2\sum_{k=1}^{n}\log^{+}(a_k)}{d\left(B_{1}\cdots B_{n}\left(i\right),i\right)}\leq\alpha\right\} , & \alpha<\alpha_{0}.
\end{cases}
\]

The  following facts will be proved in Section \ref{section:facts}.
\begin{fact}
\label{fact2} The subset of the limit set
\[
L_{\textrm{\emph{c}}}=\left\{ x\in L_{r}(G)\mid a_{k}(x)\le1\mbox{ for all }k\in\N\right\},
\]
encoding those geodesic rays whose maximal winding around any given
cusp is at most one, is contained in $\mathcal{F}_{0}$. In particular, where  $H_{0}\coloneqq\left\{ h_{i},h_{i}^{-1}\mid1\le i\le u\right\} $
denotes the symmetric set of hyperbolic generators, the set of limit points $L\left(\left\langle H_{0}\right\rangle \right)$
that can be coded exclusively via hyperbolic elements is contained
in $\mathcal{F}_{0}$. That is,
\[
L\left(\left\langle H_{0}\right\rangle \right)\subset L_{\textrm{\emph{c}}}\subset\mathcal{F}_{0}.
\]

\end{fact}
\begin{fact}
\label{fact3} We have that the set  $\left\{ x\in L_{r}(G)\mid\lim_{n\rightarrow\infty}n^{-1}\sum_{i=1}^{n}\log a_{i}\left(x\right)=\infty\right\} $ is contained in $\mathcal{F}_{1}$.
In particular, for the Jarník set
\[
\mathcal{J}\coloneqq\left\{ x\in L_{r}\left(G\right)\mid\lim_{n\rightarrow\infty}a_{n}\left(x\right)=\infty\:\mbox{and }\lim_{n\rightarrow\infty}\frac{\log a_{n}\left(x\right)}{d\left(B_{1}\cdots B_{n}\left(i\right),i\right)}=0\right\}
\]
considered in \cite{MR2900557} we have $\mathcal{J}\subset\mathcal{F}_{1}$
and $\dim_{\mathrm{H}}\left(\mathcal{J}\right)=1/2$.
\end{fact}
\begin{fact}
\label{fact4} For every $x\in\mathcal{F}_{1}$ we have that the sequence\/
$\left(a_{n}\left(x\right)\right)_{n\in\N}$ is unbounded.
\end{fact}
\begin{fact}
\label{fact1}We have $\mathcal{F}_{\alpha}\neq\emptyset$ if and
only if $\alpha\in\left[0,1\right]$.
\end{fact}
Since these level sets are generally Lebesgue null sets, the Hausdorff dimension $\dim_{\mathrm{H}}\left(\mathcal{F}_{\alpha}\right)$
is an appropriate quantity to measure the size of the sets $\mathcal{F}_{\alpha}$.
In this paper we will give a complete analysis of the \emph{cusp-winding
spectrum}
\[
f\left(\alpha\right)\coloneqq\dim_{\mathrm{{H}}}\left(\mathcal{F}_{\alpha}\right),\qquad\alpha\in\R.
\]

Using the Thermodynamic Formalism we will be able to express the function
$f$ on $[0,1]$ implicitly in terms of the \emph{cusp-winding pressure}
\emph{function}
\[
P\left(t,\beta\right)\coloneqq\lim_{n\rightarrow\infty}\frac{1}{n}\lg\sum_{B_{1}\cdots B_{n}\in\G_{n}}\e^{-td\left(B_{1}\cdots B_{n}\left(i\right),i\right)-2\beta\sum_{k=1}^{n}\log^{+}(a_k)},\quad t,\beta\in\R,
\]
where $\G_{n}\coloneqq\left\{ B_{1}\left(x\right)\cdots B_{n}\left(x\right)\in G\mid\,\,\,x\in L_{r}\left(G\right)\right\} $.
Setting, for each $n\in \N$,
\[
Z_{n}\coloneqq\lg\sum_{B_{1}\cdots B_{n}\in\G_{n}}\e^{-td\left(B_{1}\cdots B_{n}\left(i\right),i\right)-2\beta\sum_{k=1}^{n}\log^{+}(a_k)},\,\,\,n\in\mathbb{N},
\]
we have that the sequence $\left(Z_{n}\right)_{n}$
is almost sub-additive in the sense that $Z_{n+m}\leq Z_{n}+Z_{m}+2\left|t\right|C_{0}$
for a certain constant $C_{0}>0$ which will be defined in Section \ref{subsec:Horocircles-and-basic}.
In fact, this estimate boils down to the inequality $d\left(B_{1}\cdots B_{n+m}\left(i\right),i\right)\ge d\left(B_{1}\cdots B_{n}\left(i\right),i\right)+d\left(B_{1}\cdots B_{m}\left(i\right),i\right)-2C_{0}$,
which follows from (\ref{eq:distance-additive}) below. This is a consequence of the triangle inequality and  the definition of the blocks
$B_{i}$. Hence the limit of $\left(Z_{n}/n\right)$ exists and is
equal to
\begin{equation}
\lim_{n\to\infty}\frac{Z_{n}}{n}=\inf_{k\in\mathbb{N}}\left(\frac{Z_{k}}{k}+\frac{2\left|t\right|C_{0}}{k}\right).\label{eq:Subadditive}
\end{equation}

We shall see in Lemma \ref{lem:pressure is pressure} that $P$ coincides
with the dynamically-defined topological pressure function given in
(\ref{eq:DynPressure}). By Proposition \ref{pro:exTf(beta)} we have
that for every $\beta\in\R$ there exists a unique number $t=t\left(\beta\right)$,
such that $P\left(t\left(\beta\right),\beta\right)=0$. We denote
by $\beta\mapsto t\left(\beta\right)$ the \emph{cusp-winding free
energy function} (see Fig. \ref{fig:The-arithmetic-geometric-fluctuation}).
For any real-valued convex function $g$ we let $\hat{g}\colon\R\to\R\cup\left\{ \infty\right\} $
denote the \emph{Legendre transform of $g$}, which is defined to be \[\hat{g}\left(p\right)\coloneqq\sup\left\{ cp-g(c)\colon c\in\R\right\}, \text{ for all }
p\in\R.\] Further, let $\delta=\delta(G)$ denote the exponent of convergence
of the Poincar\'{e} series of $G$. Now we are in the position to
state our main theorem.

\begin{figure}
\includegraphics[width=1\textwidth]{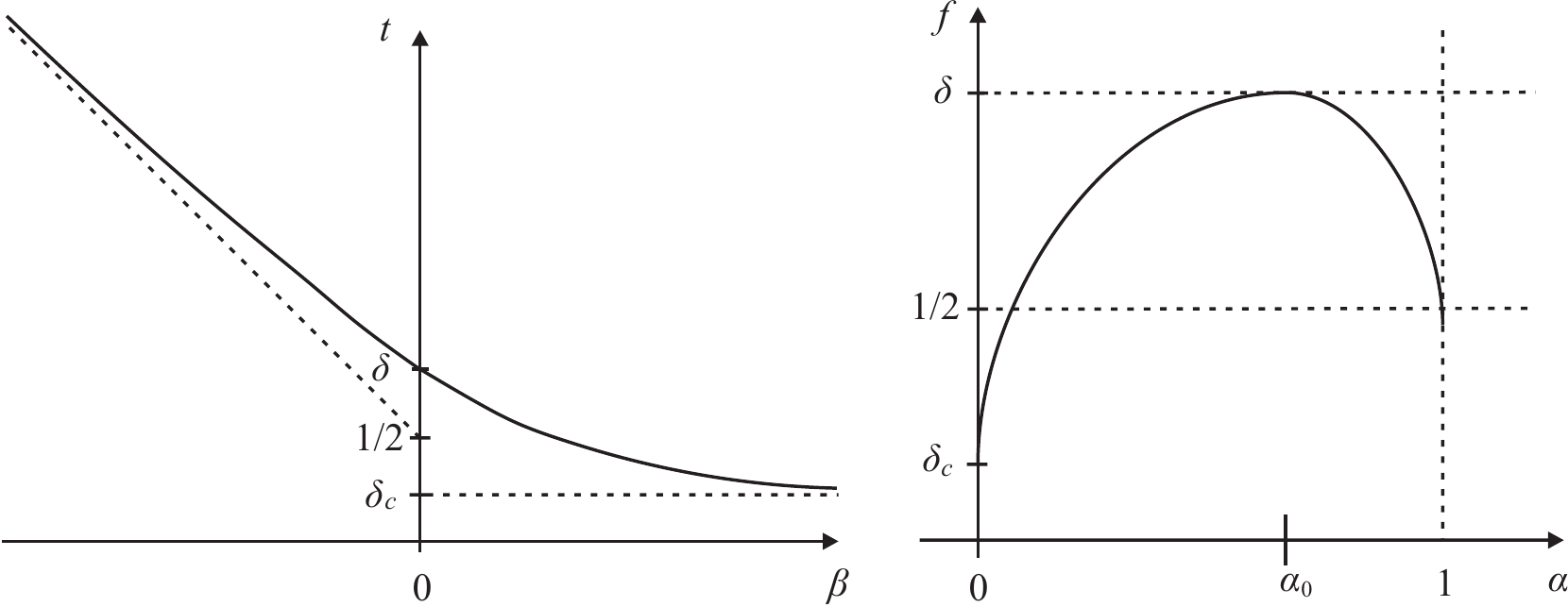}

\caption{\label{fig:The-arithmetic-geometric-fluctuation}The \emph{cusp-winding
free energy function} $\beta\protect\mapsto t\left(\beta\right)$
and the scaling spectrum $\alpha\protect\mapsto f\left(\alpha\right)$. }
\end{figure}

\begin{thm}
[Multifractal cusp-winding analysis]\label{main}For each
free, finitely generated non-elementary Fuchsian group $G$ with parabolic
elements the following statements hold. The Hausdorff dimension for the cusp
winding spectrum is given by
\[
f\left(\alpha\right)=\max\left\{ -\hat{t}\left(-\alpha\right),0\right\} ,\qquad\alpha\in\R.
\]
 \textup{\emph{The function $f\big|_{\left[0,1\right]}$ is strictly
concave, continuous, real-analytic on}}\/\textup{\emph{ $\left(0,1\right)$
and its maximal value is equal to }}\/ $\delta(G)$
(cf. Fig \ref{fig:The-arithmetic-geometric-fluctuation}). For the
boundary points we have
\[
f\left(0\right)=\delta_{c}\coloneqq\dim_{\mathrm{H}}\left(L_{c}\right),\:\:f\left(1\right)=\frac{1}{2},\:\:\mbox{and }\:\lim_{\alpha\searrow0}f'\left(\alpha\right)=-\lim_{\alpha\nearrow1}f'\left(\alpha\right)=+\infty.
\]

Further, the set of points for which the  mean cusp winding number does not exist has full Hausdorff dimension $\delta(G)$.
\end{thm}
 If we could  apply Theorem \ref{main}  to the modular surface  we would relate the cuspidal winding spectrum to the arithmetic-geometric spectrum investigated in \cite{MR2672614}. In fact, if  instead of  the modular group PSL$_{2}(\Z)$ we consider a torsion-free normal modular subgroup  $\Gamma(2)$ of index $6$, the condition  of  freeness would be achieved for $\Gamma(2)$. Then the arithmetic-geometric scaling level sets in  \cite{MR2672614} correspond to the level sets of the mean cuspidal winding for $\Gamma(2)$.

Another main result of this paper is  the multifractal formalism for level sets of quotients of Birkhoff sums in the context of conformal graph directed Markov systems (see Theorem \ref{thm:mf}). This result is obtained by extending our previous results in the context of conformal iterated function systems (\cite{MR2719683}). By using the methods of \cite{MR2719683}, we are able to develop the multifractal formalism without any  technical summability assumptions on the thermodynamic potentials, which have been imposed, for example,  in \cite[Section 4.9]{MR2003772}.

\section{Ergodic Theory for Fuchsian Groups}

In this section we study the ergodic theory for Fuchsian groups with
parabolic elements. In Section \ref{subsec:The-canonical-Markov},
we describe the action of $G$ on $\partial\H$ by a canonical Markov
map which is referred to as the Bowen--Series map.
In Section \ref{subsec:Horocircles-and-basic}, we give the crucial geometric estimates which underlie the whole multifractal  analysis.
In Section \ref{subsec:Inducing-and-the-Gibbs_Markov_Map},
we  set up  the induced
dynamics of the Bowen--Series map on the complement of some neighbourhood
of the parabolic fixed points of $G$. For the induced system we can later use Gibbs theory to prove the multifractal formalism.

\subsection{The canonical Markov map\label{subsec:The-canonical-Markov}}

As already mentioned in the introduction, throughout we exclusively
consider a finitely generated,  free  Fuchsian group $G$. Recall that $G$ can be
written as a free product $G=H\ast\Gamma$, where $H=\langle h_{1} \rangle\ast\ldots\ast\langle h_{u} \rangle$
denotes the free product of finitely many elementary, hyperbolic groups,
and $\Gamma=\langle \gamma_{1} \rangle\ast\ldots\ast\langle \gamma_{v} \rangle$ denotes the free
product of finitely many parabolic subgroups $ \langle\gamma_{i}\rangle$
 of $G$   with the parabolic fixed
point $p_{i}=\gamma_{i}^{\pm}(p_{i})$,  $i=1,\ldots,v$.  Since $G$ is finitely generated,  $G$ admits the choice of a Poincaré polyhedron $F$ with a finite
set $\mathcal{F}$ of faces. Let us now first recall from \cite{MR2158407}
the construction of the relevant coding map $T$ associated with $G$,
which maps the radial limit set $L_{r}(G)$ into itself. This construction
parallels the construction of the well-known Bowen--Series map (cf.
\cite{MR556585}, \cite{MR2098779}). For $\xi,\eta\in L_{r}(G)$,
let $\gamma_{\xi,\eta}:\R\to\H$ denote the directed geodesic from
$\eta$ to $\xi$ such that $\gamma_{\xi,\eta}$ intersects the closure
$\overline{F}$ of $F$ in $\H$, and normalised such that $\gamma_{\xi,\eta}(0)$ is the point on the geodesic from $\eta$ to $\xi$ which is closest to $i$. The exit time $e_{\xi,\eta}$
is defined by
\[
e_{\xi,\eta}\coloneqq\sup\left\{ s\mid\gamma_{\xi,\eta}(s)\in\overline{F}\right\} .
\]
 Since $\xi,\eta\in L_{r}(G)$, we clearly have that $|e_{\xi,\eta}|<\infty$.
By Poincaré's Polyhedron Theorem (cf. \cite{MR1279064}), we have
that the set $\mathcal{F}$ carries an involution $\mathcal{F}\to\mathcal{F}$,
given by $s\mapsto s'$ and $s''=s$. In particular, for each $s\in\mathcal{F}$
there is a unique \emph{face-pairing transformation} $g_{s}\in G$
such that $g_{s}(\overline{F})\cap\overline{F}=s'$. We then let
\[
\mathcal{L}_{r}(G)\coloneqq\{(\xi,\eta)\mid\xi,\eta\in L_{r}(G),\xi\neq\eta\mbox{ and }\exists\,t\in\R\;:\;\gamma_{\xi,\eta}(t)\in\overline{F}\},
\]
 and define the map $S:\mathcal{L}_{r}(G)\to\mathcal{L}_{r}(G)$,
for all $(\xi,\eta)\in\mathcal{L}_{r}(G)$ such that $\gamma_{\xi,\eta}(e_{\xi,\eta})\in s$,
for some $s\in\mathcal{F}$, by
\[
S(\xi,\eta)\coloneqq(g_{s}(\xi),g_{s}(\eta)).
\]
 In order to show that the map $S$ admits a Markov partition, we
introduce the following collection of subsets of the boundary $\partial\H$
of $\H$. For $s\in\mathcal{F}$, let $A_{s}$ refer to the open hyperbolic
half-space for which $F\subset\H\setminus A_{s}$ and $s\subset\partial A_{s}$.
We then define the projection $b_{s}$ of the side $s$ to $\partial\H$
by
\begin{equation}
b_{s}\coloneqq\mbox{Int}\left(\mbox{Cl}_{\overline{\H}}(A_{s})\cap\partial\H\right). \label{bs-def}
\end{equation}
Clearly, $b_{s}\cap b_{t}=\emptyset$, for all $s,t\in\mathcal{F},s\neq t$.
Hence, by the convexity of $F$, we have that $\gamma_{\xi,\eta}(e_{\xi,\eta})\in s$
if and only if $\xi\in b_{s}$. In other words, $S(\xi,\eta)=(g_{s}(\xi),g_{s}(\eta))$
for all $\xi\in b_{s}$. This immediately gives that the projection
map $p_{1}:(\xi,\eta)\mapsto\xi$ onto the first coordinate of $\mathcal{L}_{r}(G)$
leads to a canonical factor $T$ of $S$, that is, we obtain the map
\[
T:L_{r}(G)\to L_{r}(G),\mbox{ given by }T\arrowvert_{b_{s}\cap L_{r}(G)}\coloneqq g_{s}.
\]
 Clearly, $T$ satisfies $p_{1}\circ S=T\circ p_{1}$. Since $T(b_{s})=g_{s}(b_{s})=\mbox{Int}(\partial\H\setminus b_{s'})$,
it follows that $T$ is a non-invertible Markov map with respect to
the partition $\{b_{s}\cap L_{r}(G)\mid s\in\mathcal{F}\}$. For this
so-obtained expansive map $T$ we then have the following result.
\begin{prop}
[{\cite[Proposition 2, Proposition 3]{MR2158407}}] The map $T$
is a topologically mixing Markov map with respect to the partition
generated by $\{b_{s}\cap L_{r}(G)\mid s\in\mathcal{F}\}$. Moreover,
the map $S$ is the natural extension of $T$.
\end{prop}

\subsection{Horocircles and basic estimates\label{subsec:Horocircles-and-basic}}

\begin{figure}
\includegraphics[width=0.8\textwidth]{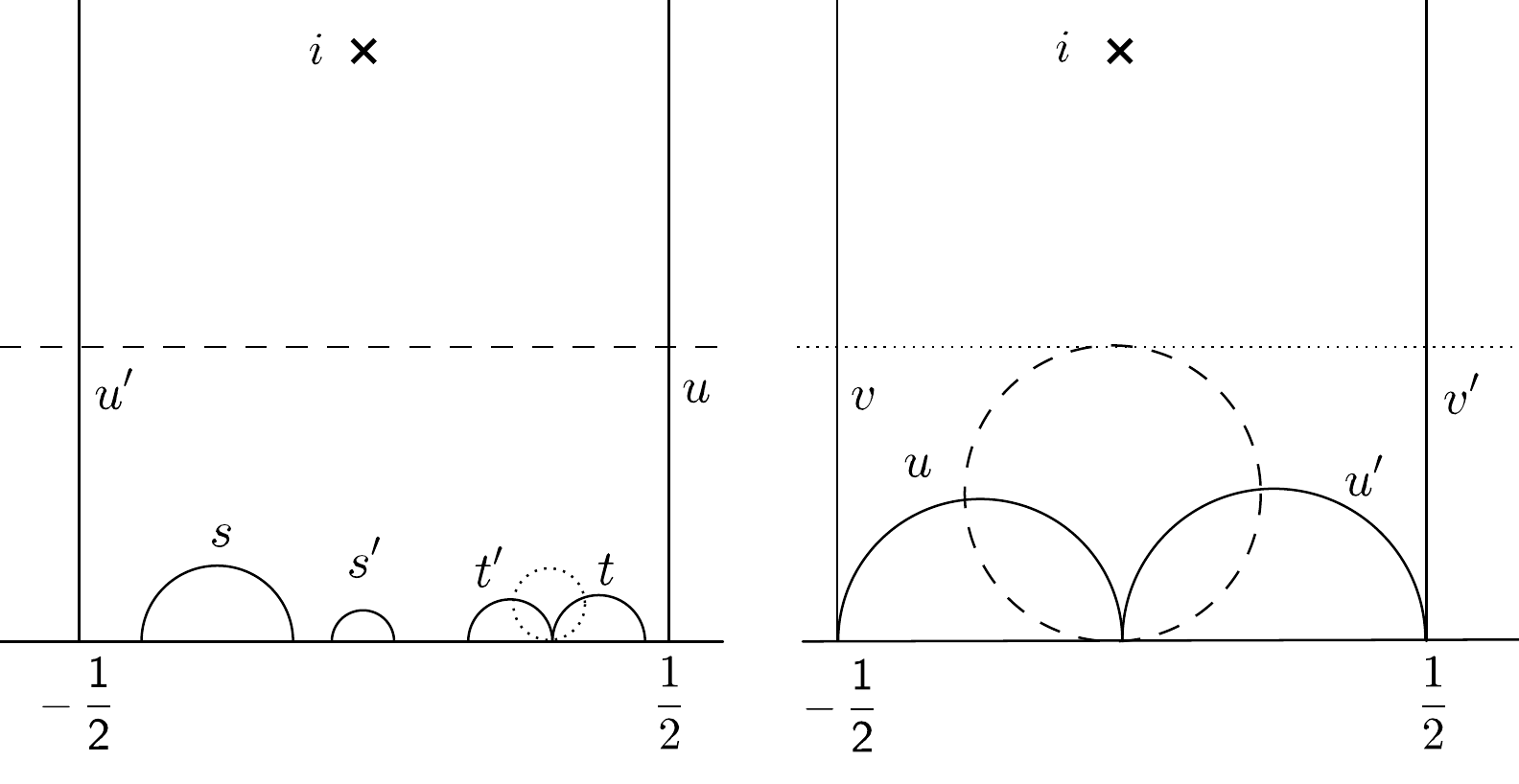}

\caption{\label{fig:Horocircle}On the left we show a typical fundamental domain
with two parabolic and one hyperbolic generator. On the right we show
the largest possible horocircle with Euclidean height $1/2$ corresponding
to the cusp at $0$ just touching another horocircle with Euclidean
height $1/2$ corresponding to the cusp at infinity. }
\end{figure}


Recall that for each $x\in L_{r}(G)$ we let $B_{k}(x)$ denote
the $k$-th block of the infinite word expansion of the geodesic ray from $i$ towards $x$ and $a_{k}(x):=\left|B_{k}\left(x\right)\right|-1$.
In order to obtain estimates for $d\left(B_{1}\cdots B_{n}\left(i\right),i\right)$,
we introduce the following horocircles. For each parabolic generator
$\gamma_{i}\in\Gamma_{0}$ with fixed point $p_{i}$, we define the
\emph{horocircle} $H_{i}$ \emph{of Euclidean height} $1/2$ which
is given as follows: there exists a unique $\Delta\in\mathrm{PSL}_{2}(\R)$
such that $\Delta(p_{i})=\infty$ and $\Delta^{-1}\gamma_{i}\Delta(z)=z+1$.
Then we set $H_{i}\coloneqq\Delta^{-1}(\{\Im(z)=1/2\})$. Another
way to define $H_{i}$ is to require that the collar geodesic, i.e.,
the projection of $H_{i}$ to $\mathbb{H}/G$, has hyperbolic length
$2$.
These horocircles are pairwise disjoint. To see this, first note that without loss of generality we can assume that one of the parabolic generators of $G$ is of the form $z\mapsto z+1$, with fixed point at infinity. The fact that $G$ is a free group ensures that the edges of the Dirichlet fundamental domain for $G$ do not intersect. An example is shown in the left-hand side of Figure  \ref{fig:Horocircle}.
In the  configuration on the right of Figure \ref{fig:Horocircle} the horocircle has  maximal Euclidean height. It is therefore sufficient to verify that the horocircles are disjoint in that case. Note that this depicts a group which is not free, to have a free group we would need to shrink the edges labelled $u$ and $u'$ so that they don't touch the vertical lines. This necessarily shrinks the Euclidean height of the horocircle, as otherwise the length of the projected collar geodesic would be greater than 2.
Let $\Delta\in\mathrm{PSL}_{2}(\R)$
be given by $\Delta(z)\coloneqq-1/(4z)$. Then the geodesic $u$ in
Figure \ref{fig:Horocircle} (right) is mapped to the vertical line
through $\Delta(-1/2)=1/2$. Similarly, the geodesic $u'$ is mapped
to the vertical line through $\Delta(1/2)=-1/2$. Hence, the parabolic
generator $\gamma$ corresponding to the side-pairing of $u$ and
$u'$ satisfies $\Delta^{-1}\gamma\Delta(z)=z+1$ and the dashed horocircle
through $0$ is represented by the horizontal line through $\Delta(i/2)=i/2$.
We have thus shown that the dashed horocircle through zero has Euclidean
height $1/2$.

To estimate $d\left(B_{1}\cdots B_{n}\left(i\right),i\right)$,  we will
partition the directed geodesic segment $\xi$ which goes from $i$ to $B_{1}\cdots B_{n}\left(i\right)$ and
corresponds to $x\in L_{r}\left(G\right)$ into $(n+1)$ arcs $\xi_{1},\dots,\xi_{n+1}$
as follows. The first arc $\xi_{1}$ starts at $i$ and the last arc
$\xi_{n+1}$ terminates at $B_{1}\cdots B_{n}\left(i\right)$. For
$k=1,\ldots,n$, the arc $\xi_{k}$ terminates at the first intersection
of $\xi$ and an image of a side of the Dirichlet fundamental domain
corresponding to $G$ centred in $i$, whenever $\left|B_{k}\right|=1$,
or the intersection of $\xi$ with the first horocircle $H_{j}$,
$j=1,\ldots,\nu$, when leaving this horocircle. The terminating point
of $\xi_{k}$ coincides with the initial point of the following arc $\xi_{k+1}$.

If we cut off the cusps of $\mathbb{H}/G$ along the horocircles $H_{1},\dots,H_{\nu}$,
then we obtain a compact manifold which in particular has a finite
diameter $C_{0}>0$. Since $\xi_{n+1}$ is contained in this compact
set we have for all $n\in\N$ that,

\begin{figure}[H]
\includegraphics[width=0.7\textwidth]{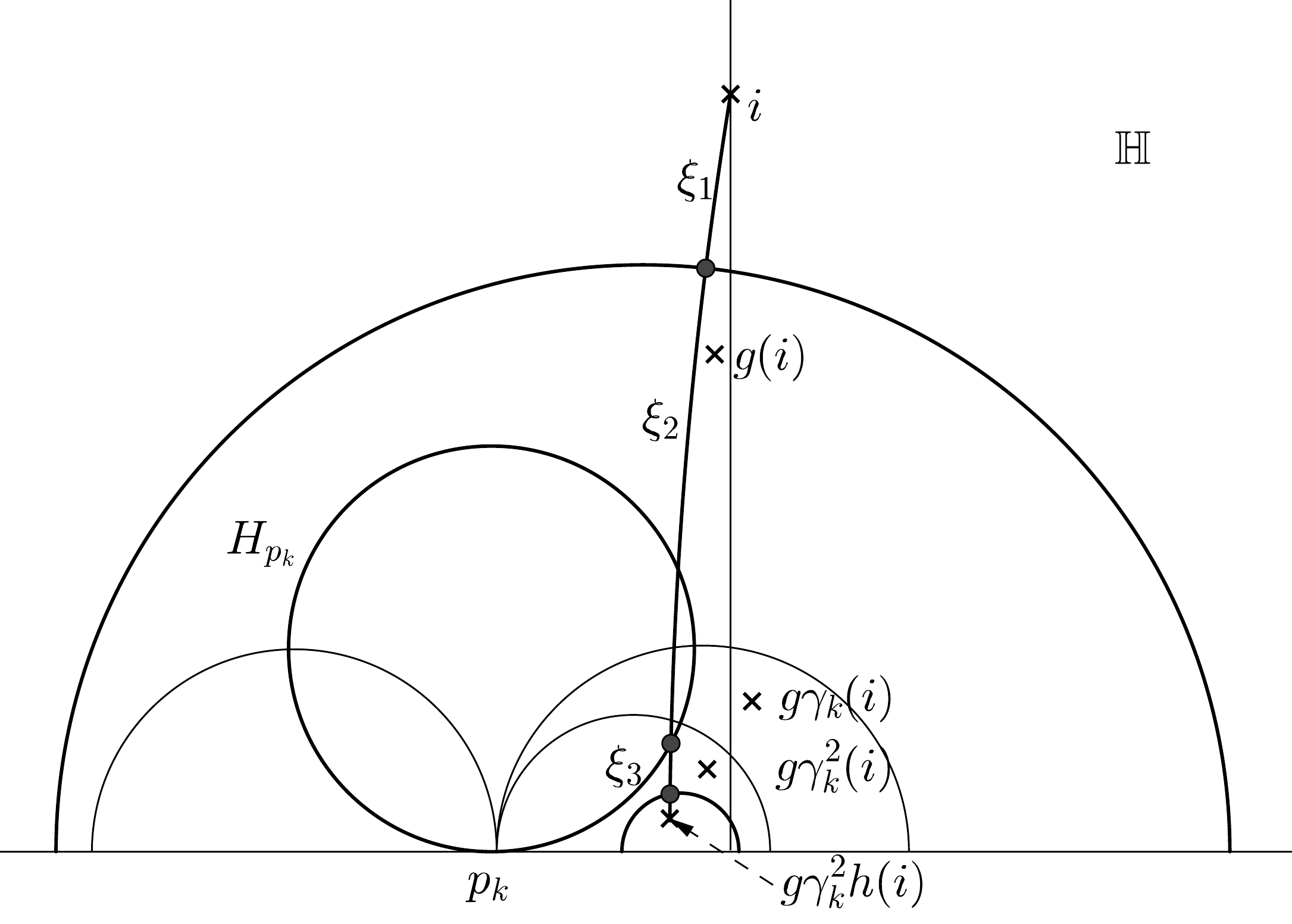}

\caption{The first three geodesic arcs $\xi_{1},\xi_{2},\xi_{3}$ connecting
$i$ and $g\gamma_{k}^{2}h\left(i\right)$, $\gamma_{k}\in\Gamma_{0}$,
$g,h\in H_{0}$. \label{geodesicArches}}
\end{figure}

\begin{equation}
\sum_{k=1}^{n}l(\xi_{k})\le d(B_{1}\cdots B_{n}(i),i)=\sum_{k=1}^{n+1}l(\xi_{k})\le\sum_{k=1}^{n}l(\xi_{k})+C_{0},\label{eq:distance-additive}
\end{equation}
 where $l(\xi_{k})$ refers to the hyperbolic length of the arc $\xi_{k}$ (see Figure \ref{geodesicArches}).

Before proving the facts stated in the introduction we will need the
following geometric observation.

\begin{figure}[H]
\includegraphics[width=0.6\textwidth]{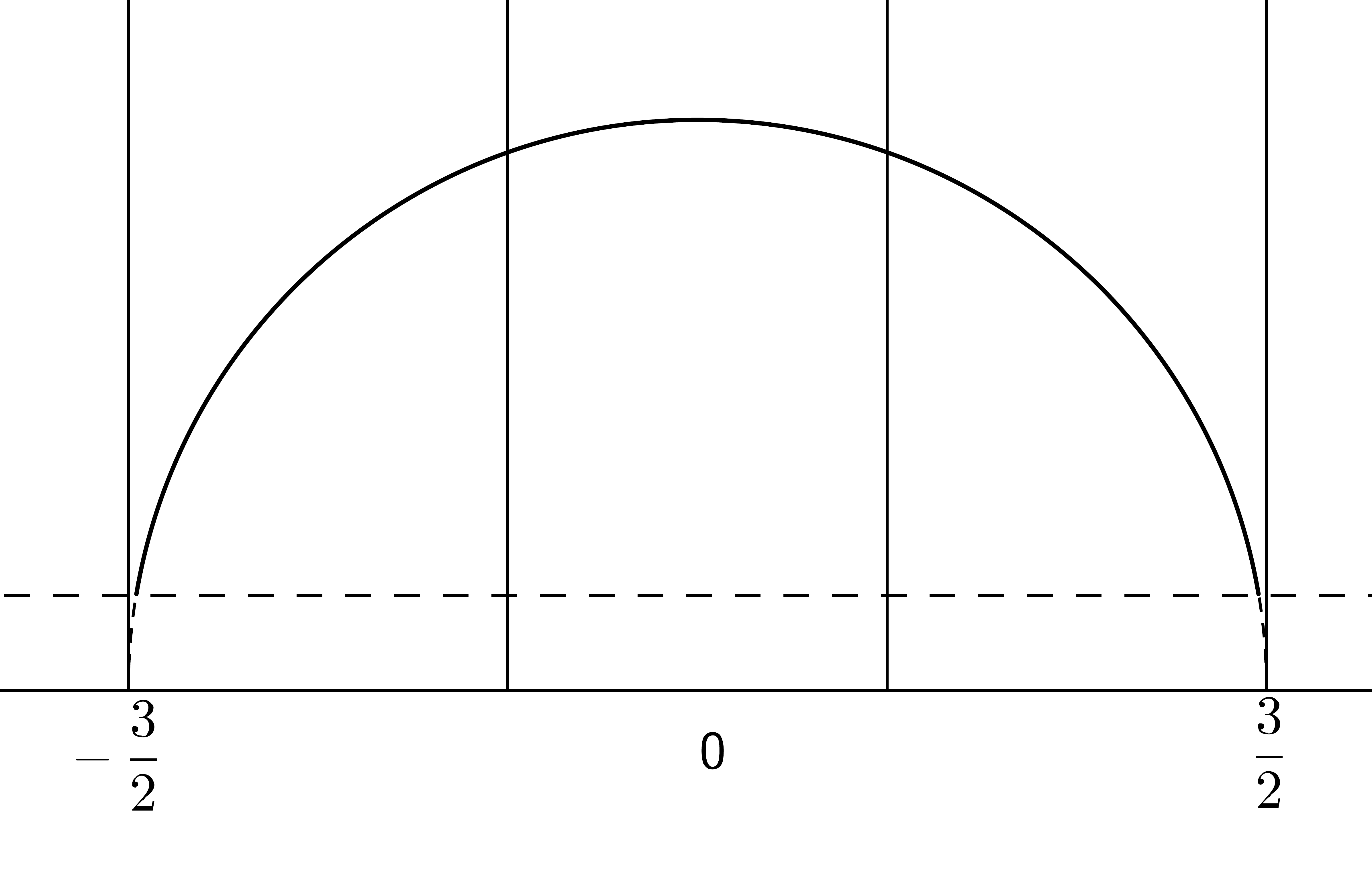}

\caption{\label{geometricObservation} The longest possible parabolic excursion above $\left\{ \Im\left(z\right)=1/2\right\} $
(dotted line) corresponding to  $n=4$ in Proposition
\ref{prop:parabolic_excursion_length}.}
\end{figure}

\begin{prop}
\label{prop:parabolic_excursion_length} For $n\ge3$ we have that
the length of the geodesic arc $\xi$ lying above $\left\{ \Im\left(z\right)=1/2\right\} $
and connecting $\pm(n-1)/2$ (cf. Figure \ref{geometricObservation}) lies between the two constants $\log\left(3\cdot\left(n-1\right)^{2}\right)$
and $\log\left(4\cdot\left(n-1\right)^{2}\right)$.
\end{prop}
\begin{proof}
Let $0<h<r$ and put $a\coloneqq\sqrt{r^{2}-h^{2}}$. We parametrise
the circular arc of the circle with centre $0$ and radius $r$ which
lies above the line $\left\{ z\in\H\mid\Im(z)=h\right\} $. The parametrisation
is given by $\gamma:\left[-a,a\right]\rightarrow\H$, $\gamma(t)\coloneqq\left(t,\sqrt{r^{2}-t^{2}}\right)$.
We have
\[
\gamma'(t)=\left(1,-t(r^{2}-t^{2})^{-1/2}\right)\quad\mbox{and }\Vert\gamma'(t)\Vert_{2}=\sqrt{1+\frac{t^{2}}{r^{2}-t^{2}}}=\frac{r}{\sqrt{r^{2}-t^{2}}}.
\]
Hence, the hyperbolic length of $\gamma$ is given by
\[
l(\gamma)=\int_{-a}^{a}\Vert\gamma'(t)\Vert_{2}\frac{1}{\sqrt{r^{2}-t^{2}}}dt=\frac{1}{r}\int_{-a}^{a}\frac{1}{1-(t/r)^{2}}dt=\int_{-a/r}^{a/r}\frac{1}{1-y^{2}}dy=\log\left(\frac{1+a/r}{1-a/r}\right).
\]
It will be convenient to write
\[
l(\gamma)=\log\left(\frac{1+a/r}{1-a/r}\right)=\log\left(\frac{\left(1+a/r\right)^{2}}{1-\left(a/r\right)^{2}}\right).
\]
To verify the statement in the proposition,  we must consider the case that $h\coloneqq1/2$ and $r\coloneqq(n-1)/2$.
Then $a=\sqrt{n^{2}-2n}/2$ and by our previous calculation we obtain
\[
l\left(\xi\right)=\log\left(\frac{\left(1+\frac{\sqrt{n^{2}-2n}}{n-1}\right)^{2}}{1-\frac{n^{2}-2n}{\left(n-1\right)^{2}}}\right)=\log\left((n-1)^{2}\left(1+\frac{\sqrt{n^{2}-2n}}{n-1}\right)^{2}\right).
\]
To complete the proof, observe that, for all $n\geq3$, we have $4\geq\left(1+\frac{\sqrt{n^{2}-2n}}{n-1}\right)^{2}\ge3.$
\end{proof}
\begin{cor}
\label{cor:winding estimate}For any geodesic arc $\xi_{i}$ corresponding
to a given block $B_{i}$ with $a_{i}\ge2$, we have
\begin{equation}
2\log\left(a_{i}\right)+\log3\leq l\left(\xi_{i}\right).\label{eq:lower_bound_a_i>3}
\end{equation}
If $a_{i}\ge0$ then we still have
\begin{equation}
2\log^{+}(a_{i})+C_{1}\le l(\xi_{i})\le2\log(a_{i}+1)+C_{2},\label{eq:lower_and_upper_bound_a_i>1}
\end{equation}
where $C_{1}\coloneqq\min\left\{ \log3,d\left(s,t\right)\mid s,t\in\mathcal{F},g_{s}g_{t}\not=\mbox{id.}\right\} $
and $C_{2}\coloneqq\max\left\{ \log4,C_{0}\right\}$, where $C_{0}$ as defined in Section \ref{subsec:Horocircles-and-basic}.
\end{cor}

\subsection{Inducing and the topological Markov chain with infinite state space\label{subsec:Inducing-and-the-Gibbs_Markov_Map}}
Let us set  \[\mathcal{D}\coloneqq\left\{ y\in L_{r}(G)\mid a_{1}(y)=0\right\}.\]
The induced transformation $T_{\mathcal{D}}$ on $\mathcal{D}=\{a_{1}=0\}$
is defined by $T_{\mathcal{D}}(\xi)\coloneqq T^{\rho(\xi)}(\xi)$,
where $\rho$ denotes the return time function, given by $\rho(\xi)\coloneqq\min\{n\in\N\mid T^{n}(\xi)\in\mathcal{D}\}$.
Denote by $\Gamma_{0}\coloneqq\left\{ \gamma_{i},\gamma_{i}^{-1}\mid1\le i\le\nu\right\} $
the symmetric set of parabolic generators. Recall that $G_{0}=\Gamma_{0}\cup H_{0}$.
Define the induced partition
\[
\alpha_{\mathcal{D}}\coloneqq\bigcup_{n\in\N}\mathcal{D}\cap\left\{ \rho=n\right\} \cap\bigvee_{k=0}^{n+1}T^{-k}\left(\left\{ b_{s}\cap L_{r}(G)\mid s\in\mathcal{F}\right\} \right)
\]
and the infinite alphabet
\begin{align*}
I & \coloneqq\left\{ g_{1}\gamma^{n}g_{2}\mid\gamma\in\Gamma_{0},g_{1},g_{2}\in G_{0}\setminus\{\gamma^{\pm1}\},n\in\N\right\} \cup\\
 & \quad\quad\cup\left\{ g_{1}hg_{2}\mid g_{1}\in G_{0},h\in H_{0}\setminus\{g_{1}^{-1}\},g_{2}\in G_{0}\setminus\{h^{-1}\}\right\} .
\end{align*}

The incidence matrix $A\in\{0,1\}^{I\times I}$ is for $\eta=\eta_{1}\dots\eta_{n}\in I$
and $\tau=\tau_{1}\dots\tau_{m}\in I$, $n,m\geq3$, given by $A(\eta,\tau)=1$
if and only if $\eta_{n-1}=\tau_{1}$ and $\eta_{n}=\tau_{2}$. We
consider the Markov shift $(\Sigma_{A},\sigma)$ where
\[
\Sigma_{A}\coloneqq\left\{ \omega=(\omega_{1},\omega_{2},\dots)\in I^{\N}\mid\forall k\in\N\,\,\,A(\omega_{k},\omega_{k+1})=1\right\}
\]
and $\sigma\colon\Sigma_{A}\to\Sigma_{A}$, $\left(\sigma\left(\omega\right)\right)_{i}=\omega_{i+1}$
for $\omega\in\Sigma_{A}$ denotes the \emph{left shift map.}

Recall that $\mathcal{D}:=\{a_1=0\}$. The induced Bowen--Series map $(\mathcal{D},T_{\mathcal{D}})$ is conjugated
to the left shift $(\Sigma_{A},\sigma)$ via the coding map $\pi:\mathcal{D}\rightarrow \Sigma_{A}$ which is defined in terms of the infinite word expansion, as given in the introduction, of limit points. Note that $\pi$ is surjective, because any limit point with expansion having first block of length 1 necessarily corresponds to a word starting with a letter in $I$. That is, we have
the following commutative diagram
\[
\xymatrix{\mathcal{D}\ar[r]^{T_{\mathcal{D}}}\ar[d]_{\pi} & \mathcal{D}\ar[d]^{\pi}\\
\Sigma_{A}\ar[r]_{\sigma} & \Sigma_{A}
}
\]

For $\omega\in\Sigma_{A}$ and $n\in\N$ let $\omega|_{n}\coloneqq\left(\omega_{1},\dots,\omega_{n}\right)$
and let $[\omega|_{n}]\coloneqq\left\{ \tau\in\Sigma_{A}\mid\tau_{1}=\omega_{1},\dots,\tau_{n}=\omega_{n}\right\} $
denote the $n$-\emph{cylinder} of $\omega$. We have $\pi(\mathcal{D}) =\left\{ [\omega]\mid\omega\in I \right\}$.
%

We now aim to express the cusp-winding scaling limit in
dynamical terms. For this we introduce the two \emph{potential functions}
\begin{align*}
\psi\colon\Sigma_{A} & \to\R_{0}^{-},\:\left(\omega_{1},\omega_{2},\ldots\right)\mapsto-2\log^{+}(|\omega_{1}|-3),\\
\varphi\colon\Sigma_{A} & \to\R^{-},\:\left(\omega_{1},\omega_{2},\ldots\right)\mapsto-\log\left(|\left(g_{1}\cdots g_{|\omega_{1}|-2}\right)'\left(\pi^{-1}\left(\omega\right)\right)|\right),\\
 & \qquad\qquad\qquad\qquad\hfill\,\mbox{with }\omega_{1}=g_{1}\cdots g_{|\omega_{1}|}\in I,
\end{align*}
where $\psi$ describes the cusp-winding number, while $\varphi$
describes the geometric properties of the geodesic flow. We will equip
$\Sigma_{A}$ with the metric $d$ which is given for each $\omega,\tau\in I^{\N}$
 by $d(\omega,\tau)\coloneqq\exp\left(-|\omega\wedge\tau|\right),$
where $|\omega\wedge\tau|$ denotes the length of the longest common
initial block $\omega\wedge\tau$ of $\omega$ and $\tau$. Since $\psi$ depends only
on the first symbol, we immediately see that $\psi$ is Hölder
continuous with respect to this metric. For the proof that $\varphi$
is also Hölder continuous we refer to \cite{MR2041265}.

\subsection{Topological pressure}

First, let us fix some notation. For two sequences $\left(a_{n}\right)$, $\left(b_{n}\right)$ we will write $a_{n}\ll b_{n}$, if $a_{n}\le Kb_{n}$ for some $K>0$
and all $n\in\N$, and if $a_{n}\ll b_{n}$ and $b_{n}\ll a_{n}$
then we write $a_{n}\asymp b_{n}$.

For $n\in\N$ put
\[
\Sigma_{A}^{n}=\left\{ \omega=(\omega_{1},\dots,\omega_{n})\in I^{n}\mid A(\omega_{i},\omega_{i+1})=1\mbox{ for all }1\le i\le n-1\right\} .
\]
The topological pressure $\mathfrak{P}\left(t\varphi+\beta\psi\right)$
of the potential $t\varphi+\beta\psi$ for $t,\beta\in\R$ is defined
to be
\begin{equation}
\mathfrak{P}\left(t\varphi+\beta\psi\right)\coloneqq\lim_{n\rightarrow\infty}\frac{1}{n}\lg\sum_{\omega\in\Sigma_{A}^{n}}\exp\sup_{\tau\in[\omega]}\left(S_{\omega}\left(t\varphi+\beta\psi\right)\right)(\tau),\label{eq:DynPressure}
\end{equation}
where we set $S_{\omega}f\coloneqq\sup_{x\in[\omega]}S_{n}f\left(x\right)$
with $S_{n}f\left(x\right)\coloneqq\sum_{k=0}^{n-1}f\circ\sigma^{k}\left(x\right)$
and $S_{0}f=0$. By a standard argument involving sub-additivity the
above limit always exists (although it is possibly equal to infinity).

The next lemma shows that the set $\mathcal{F}_{\alpha}$ can be
characterised by the potentials $\varphi$ and $\psi$ and that the
cusp-winding pressure $P$$\left(t,\beta\right)$ agrees with $\mathfrak{P}\left(t\varphi+\beta\psi\right)$
for all $t,\beta\in\R$.
In the proof of the following lemma we make use of the fact that  the topological Markov chain $(\Sigma_{A},\sigma)$ is  \emph{finitely primitive}, that is,  there exists $\ell\in \N$ and a finite set $F\subset \Sigma_A^{\ell}$  such that for all $a,b\in I$ there exists $u\in  \Sigma_A^{\ell}$ such that the word $aub$ belongs to $  \Sigma_A^{\ell+2}$, for more details, see \cite{MR2003772}. Here the finite set $F$ can be constructed due to the fact that there are only finitely many generators of $G$.
\begin{lem}
\label{lem:pressure is pressure} For $\alpha\in\R$, $x\in \mathcal{D}$
and $\omega\coloneqq\pi(x)\in\Sigma_{A}$ we have
\[
\lim_{n\rightarrow\infty}\frac{S_{n}\psi\left(\omega\right)}{S_{n}\varphi\left(\omega\right)}=\alpha\iff\lim_{n\rightarrow\infty}\frac{2\sum_{i=1}^{n}\log(a_{i}(x))}{d\left(B_{1}\dots B_{n}\left(i\right),i\right)}=\alpha,
\]
and for all $t,\beta\in\R$ we have
\[
P\left(t,\beta\right)=\mathfrak{P}\left(t\varphi+\beta\psi\right).
\]
Moreover, $\mathfrak{P}\left(t\varphi+\beta\psi\right)<\infty$ if
and only if $t+\beta>1/2$. Further, for each $\beta\in\R$ we have
$\lim_{t\searrow\frac{1}{2}-\beta}P\left(t,\beta\right)=\infty$ and
$\lim_{t\to\infty}P\left(t,\beta\right)=-\infty$.
\end{lem}
\begin{proof}
Following \cite{MR2041265}, we find a constant $D>0$ such that for
all $x\in L_{r}\left(G\right)$, $\omega\coloneqq\pi(x)\in\Sigma_{A}$
and $n\in\N$,
\begin{equation}
\left|d\left(B_{1}\dots B_{n}\left(i\right),i\right)-S_{n}\varphi\left(\omega\right)\right|\le D.\label{eq:4}
\end{equation}
 This proves the first two assertions.

Since  $(\Sigma_{A},\sigma)$ is finitely
primitive, it follows from  \cite[Proposition 2.1.9]{MR2003772}
that
\[
\mathfrak{P}\left(t\varphi+\beta\psi\right)<\infty\quad\mbox{if and only if }\quad\sum_{a\in I}\e^{\sup_{x\in\left[a\right]}t\varphi\left(x\right)+\beta\psi\left(x\right)}<\infty.
\]
Now, we deduce with (\ref{eq:4}), (\ref{eq:distance-additive}) and
(\ref{eq:lower_and_upper_bound_a_i>1}) that for all $t,\beta\in\mathbb{R}$
\begin{equation}
\sum_{a\in I}\e^{\sup_{x\in\left[a\right]}t\varphi\left(x\right)+\beta\psi\left(x\right)}\asymp\sum_{n\ge1}\e^{-2t\log\left(n\right)-2\beta\log(n)}\asymp\sum_{n\ge1}n^{-2(\beta+t)}.\label{eq:PressureDiverging}
\end{equation}
which converges if and only if $t+\beta>1/2$. This also shows $\lim_{t\searrow\frac{1}{2}-\beta}P\left(t,\beta\right)=\infty$.

Finally, by (\ref{eq:Subadditive}) and Corollary \ref{cor:winding estimate}
(with the constant $C_{1}$ as defined in Corollary \ref{cor:winding estimate}),
we have for fixed $n>C_{1}/(2C_{0})$ and $t>0$,

\begin{align*}
P\left(t,\beta\right) & \leq\frac{1}{n}\log\sum_{B_{1}\cdots B_{n}\in\G_{n}}\e^{-td\left(B_{1}\cdots B_{n}\left(i\right),i\right)-2\beta\sum_{i=1}^{n}\log^{+}(|B_{i}|-1)}+\frac{2tC_{0}}{n}\\
 & \leq\frac{1}{n}\log\sum_{B_{1}\cdots B_{n}\in\G_{n}}\e^{-t\sum_{i=1}^{n}\left(2\log^{+}(a_{i})+C_{1}\right)-2\beta\sum_{i=1}^{n}\log^{+}(a_{i})}+\frac{2tC_{0}}{n}\\
 & \leq t\underbrace{\left(\frac{2C_{0}}{n}-C_{1}\right)}_{<0}+\log\left(2\nu\sum_{k=1}^{\infty}k^{-2\left(t+\beta\right)}+2u\right)\to-\infty\,\,\mbox{for }t\to\infty.
\end{align*}
\end{proof}

\subsection{The cusp-winding  free energy }

The following proposition shows the existence of the cusp-winding
free energy function and some of its basic properties.
\begin{prop}
\label{pro:exTf(beta)}For each $\beta\in\R$ there exists a unique
number $t\left(\beta\right)$ such that
\begin{equation}
P\left(t\left(\beta\right),\beta\right)=0.\label{eq:pressure equation in gauss}
\end{equation}
 The cusp-winding free energy function $t:\mathbb{R}\to\mathbb{R}$
defined in this way is real-analytic and strictly convex.
\end{prop}
\begin{proof}
We have to verify that the function  $(t,\beta)\mapsto P(t,\beta)$ is real-analytic
on the set $\left\{ (t,\beta)\in\R^{2}\mid2(t+\beta)>1\right\} $. Recall
that $\Sigma_{A}$ is a finitely primitive Markov shift. Let $(t,\beta)\in\R^{2}$
with $2(t+\beta)>1$. By \cite[Proposition 2.6.13]{MR2003772} it
suffices to show that $\varphi,\psi\in L^1(\Sigma_{A},\mu_{t\varphi+\beta\psi})$, where $L^1(\Sigma_{A},\mu_{t\varphi+\beta\psi})$ denotes the space of $\mu_{t\varphi+\beta\psi}$-integrable functions on $\Sigma_{A}$.
Here, $\mu_{t\varphi+\beta\psi}$ denotes the unique equilibrium state
of $t\varphi+\beta\psi$ for the dynamical system $(\Sigma_{A},\sigma)$,
that is, $\mu_{t\varphi+\beta\psi}$ is the unique $\sigma$-invariant
Borel probability measure on $\Sigma_{A}$ such that $h(\mu_{t\varphi+\beta\psi})+\int t\varphi+\beta\psi\,\,d\mu_{t\varphi+\beta\psi}=0$,
where $h(\mu_{t\varphi+\beta\psi})$ refers to the metric entropy
of the measure-theoretic dynamical system $(\Sigma_{A},\sigma,\mu_{t\varphi+\beta\psi})$.
It is well known that $\mu_{t\varphi+\beta\psi}$ is a Gibbs measure
with respect to $t\varphi+\beta\psi$. In particular, using the estimate
in (\ref{eq:PressureDiverging}), we have that $\mu_{t\varphi+\beta\psi}([\omega_{1}])\ll\exp\left(-2t\log\left|\omega_{1}\right|-2\beta\log\left|\omega_{1}\right|\right)$
for all $\omega_{1}\in I$. Since $2(t+\beta)>1$, it follows that
$\varphi,\psi\in L^1(\Sigma_{A},\mu_{t\varphi+\beta\psi})$.
By \cite[Theorem 2.6.12]{MR2003772} we know that the pressure $P$
is real-analytic on $\left\{ \left(t,\beta\right)\in\R^{2}\mid P\left(t,\beta\right)<\infty\right\} $,
which by Lemma \ref{lem:pressure is pressure} is equal to $\left\{ (t,\beta)\in\R^{2}\mid2(t+\beta)>1\right\} $.

Now let $\beta\in\R$. By Lemma \ref{lem:pressure is pressure} we
have that $P\left(t,\beta\right)<\infty$, if and only if $t>1/2-\beta$.
Also by Lemma \ref{lem:pressure is pressure}, we have $\lim_{t\searrow\frac{1}{2}-\beta}P\left(t,\beta\right)=\infty$
and $\lim_{t\to\infty}P\left(t,\beta\right)=-\infty$. Hence, we conclude
that there exists a unique $t=t\left(\beta\right)$ with $P\left(t\left(\beta\right),\beta\right)=0$.

The strict convexity of $t$ follows from convexity and real-analyticity,
unless $t'$ is constant. That $t'$ is not constant follows from
the asymptotic behaviour of $t$ for $\beta\to\pm\infty$ as described
in the following Lemmas \ref{lem:Asymp_t_right} and \ref{lem:tAsymptoticFor-right}.
\end{proof}
\begin{lem}
\label{lem:Asymp_t_right}For the right asymptotic of $t$ we have that
$\lim_{\beta\to\infty}t\left(\beta\right)=\delta_{c}$.
\end{lem}
\begin{proof}
For $s\in\mathbb{R}$ and $n\in\N$, set
\[
p_{n}\left(s\right)\coloneqq\frac{1}{n}\log\sum_{{B_{1}\cdots B_{n}\in\G_{n}\atop \left|B_{i}\right|\leq2}}\e^{-s d\left(B_{1}\cdots B_{n}i,i\right)}\,\,\mbox{and }p\left(s\right)\coloneqq\lim_{n\rightarrow \infty}p_{n}\left(s\right).
\]
Then for all $\beta>0$, we have
\begin{align*}
0 & =\lim_{n\to\infty}\frac{1}{n}\log\sum_{B_{1}\cdots B_{n}\in\G_{n}}\e^{-t\left(\beta\right)d\left(B_{1}\cdots B_{n}i,i\right)-2\beta\sum\log^{+}\left(a_{i}\right)}
\geq  \lim_{n\to\infty} p_{n}(t(\beta)) 
=p\left(t\left(\beta\right)\right).
\end{align*}
 By the formula of Bishop and Jones (\cite{MR1484767}) it follows that $t\left(\beta\right)\geq\delta_{c}$.

Now, let $C\coloneqq2C_{0}$ with $C_{0}$ as in Section \ref{subsec:Horocircles-and-basic},
 and recall that the constant $\nu$
denotes the number of cusps. Let $\zeta$ denote the Riemann zeta-function
and choose $\epsilon>0$. Then by using the triangle inequality $3\left(n-k\right)$
times corresponding to blocks of length greater than 4, we obtain
\begin{align*}
 & \lefteqn{\frac{1}{n}\log\sum_{B_{1}\cdots B_{n}\in\G_{n}}\e^{-\left(\delta_{c}+\epsilon\right)d\left(B_{1}\cdots B_{n}i,i\right)-2\beta\sum\log^{+}\left(a_{i}\right)}}\\
 & \leq\frac{1}{n}\log\sum_{k=0}^{n}\binom{n}{k}\e^{kp_{k}\left(\delta_{c}+\epsilon\right)}\e^{3\left(\delta_{c}+\epsilon\right)C\left(n-k\right)}\sum_{{\left(B_{1},\ldots,B_{n-k}\right)\in\G_{1}^{n-k}\atop \left|B_{i}\right|\geq3}}\e^{\sum_{\ell=1}^{n-k}-\left(\delta_{c}+\epsilon\right)d\left(B_{\ell}i,i\right)-2\beta\log^{+}\left(a_{i}\right)}\\
 & \leq\frac{1}{n}\log\sum_{k=0}^{n}\binom{n}{k}\e^{kp_{k}\left(\delta_{c}+\epsilon\right)}\e^{3\left(\delta_{c}+\epsilon\right)C\left(n-k\right)}\sum_{{\left(B_{1},\ldots,B_{n-k}\right)\in\G_{1}^{n-k}\atop \left|B_{i}\right|\geq3}}\prod_{i=1}^{n-k}a_{i}^{-2\beta}\\
 & \leq\frac{1}{n}\log\sum_{k=0}^{n}\binom{n}{k}\e^{kp_{k}\left(\delta_{c}+\epsilon\right)}\e^{3\left(\delta_{c}+\epsilon\right)C\left(n-k\right)}\left(2\nu\left(\zeta\left(2\beta\right)-1\right)\right)^{n-k}\\
 & \stackrel{n\to\infty}{\longrightarrow}\quad\log\bigg(\underbrace{\e^{p\left(\delta_{c}+\epsilon\right)}}_{<1}+\underbrace{2\e^{3\left(\delta_{c}+\epsilon\right)C}\nu\left(\zeta\left(2\beta\right)-1\right)}_{\stackrel{\beta\to\infty}{\longrightarrow}0}\bigg)<0
\end{align*}
for $\beta$ large enough. Here we used the general observation that
for $\lambda_{k}\to\lambda>0$ and $b>0$ we have
\[
\lim_{n\to\infty}\left(\sum_{k=0}^{n}{n \choose k}\lambda_{k}^{k}b^{n-k}\right)^{1/n}=\lambda+b.
\]
This shows that for all sufficiently large $\beta$ we have $\delta_{c}\leq t\left(\beta\right)\leq\delta_{c}+\epsilon.$
Since $\epsilon>0$ was arbitrary, it follows that $\lim_{\beta\to\infty}t\left(\beta\right)=\delta_{c}$.
\end{proof}
\begin{lem}
\label{lem:tAsymptoticFor-right}For all $\epsilon\in\left(0,1/2\right)$
we find that for all $-\beta$ large enough we have
\[
\frac{1}{2}-\beta<t\left(\beta\right)<\frac{1}{2}-\beta+\frac{\epsilon}{2}.
\]
\end{lem}
\begin{proof}
Combining Lemma \ref{lem:pressure is pressure} and Proposition \ref{pro:exTf(beta)}
we deduce that $t\left(\beta\right)\geq1/2-\beta$.

For the proof of the upper bound it suffices to show that for a given $\epsilon\in\left(0,1/2\right)$
and $t_{\beta}\coloneqq1/2-\beta+\epsilon/2$ we have $P\left(t_{\beta},\beta\right)<0$
for all $-\beta$ large. Recall once more that $\nu$ denotes the number of
cusps and $u=\card\left(H_{0}\right)$. Fix a natural number $n>2C_{0}/C_{1}$
with $C_{0}>0$ taken from (\ref{eq:Subadditive}) and $C_{1}$ taken
from (\ref{eq:lower_and_upper_bound_a_i>1}). By (\ref{eq:Subadditive})
and (\ref{eq:lower_and_upper_bound_a_i>1}) combined with (\ref{eq:distance-additive})
we have for all $\beta$ negative
\begin{align*}
P\left(t_{\beta},\beta\right) & \leq\frac{1}{n}\log\sum_{B_{1}\cdots B_{n}\in\G_{n}}\e^{-t_{\beta}d\left(B_{1}\cdots B_{n}\left(i\right),i\right)-2\beta\sum_{i=1}^{n}\log^{+}(|B_{i}|-1)}+\frac{2t_{\beta}C_{0}}{n}\\
 & \leq\frac{1}{n}\log\sum_{B_{1}\cdots B_{n}\in\G_{n}}\e^{-t_{\beta}\sum_{i=1}^{n}\left(2\log^{+}(a_{i})+C_{1}\right)-2\beta\sum_{i=1}^{n}\log^{+}(a_{i})}+\frac{2t_{\beta}C_{0}}{n}\\
 & \leq t_{\beta}\underbrace{\left(\frac{2C_{0}}{n}-C_{1}\right)}_{<0}+\frac{1}{n}\log\left(\left(2\nu\sum_{k=1}^{\infty}k^{-\left(1+\epsilon\right)}+2u\right)^{n}\right)\to-\infty\,\,\mbox{for }\beta\to-\infty.
\end{align*}
\end{proof}

\section{Proof of facts} \label{section:facts}
\begin{proof}
[Proof of Fact \ref{fact2}] The inclusion $L_{c}\subset\mathcal{F}_{0}$
follows immediately from the definition of $\mathcal{F}_{0}$.
\end{proof}

\begin{proof}
[Proof of Fact \ref{fact3}]  Using (\ref{eq:distance-additive})
and (\ref{eq:lower_and_upper_bound_a_i>1}) along with the property
that $\log(k+1)\le\log^{+}(k)+1$, for $k\in\mathbb{N}_{0}$, we have
\begin{align*}
\liminf_{n\rightarrow\infty}\frac{2\sum_{i=1}^{n}\log^{+}\left(a_{i}(x)\right)}{d\left(B_{1}\cdots B_{n}\left(i\right),i\right)} & \ge\liminf_{n\rightarrow\infty}\frac{\sum_{i=1}^{n}\log^{+}\left(a_{i}(x)\right)}{\sum_{i=1}^{n}\log^{+}\left(a_{i}(x)\right)+n(C_{2}+1)+C_{0}/2}\\
 & =\liminf_{n\rightarrow\infty}\left(1+\frac{n(C_{2}+1)+C_{0}/2}{\sum_{i=1}^{n}\log^{+}(a_{i}(x))}\right)^{-1}.
\end{align*}
By assumption, the Cesàro average of $\log(a_{i}(x))$ tends to infinity, which implies that
\[
\liminf_{n\rightarrow\infty}\left(1+\frac{n(C_{2}+2)+C_{0}/2}{\sum_{i=1}^{n}\log^{+}(a_{i}(x))}\right)^{-1}=1.
\]
Combining this with (\ref{eq:fact1 upper bound}) below finishes the
proof. For the proof of the equality
\[
\dim_{\mathrm{H}}\left(\mathcal{J}\right)=1/2
\]
we refer to \cite{MR2900557}.
\end{proof}

\begin{proof}
[Proof of Fact \ref{fact4}] We prove that if $a_{i}(x)\le K$
for all $i\in\N$, then $x\notin\mathcal{F}_{1}$. By (\ref{eq:lower_and_upper_bound_a_i>1})
we have for $a_{i}\geq0$
\[
2\log^{+}(a_{i})+C_{1}\le l(\xi_{i}).
\]
 Then we have by (\ref{eq:distance-additive})
\begin{align*}
\limsup_{n\rightarrow\infty}\frac{2\sum_{i=1}^{n}\log^{+}(a_{i}(x))}{d\left(B_{1}\cdots B_{n}\left(i\right),i\right)} & \le\limsup_{n\rightarrow\infty}\frac{2\sum_{i=1}^{n}\log^{+}(a_{i}(x))}{\sum_{i=1}^{n}l(\xi_{i})}\\
 & \le\limsup\frac{2\sum_{i=1}^{n}\log^{+}(a_{i}(x))}{2\sum_{i=1}^{n}\log^{+}(a_{i}(x))+nC_{1}}\\
 & \le\left(1+\frac{C_{1}}{2\log(K)}\right)^{-1}<1.
\end{align*}
\end{proof}

\begin{proof}
[Proof of Fact \ref{fact1}] It is well known that  $\left\{ \alpha\in\R\mid\mathcal{F}_{\alpha}\neq\emptyset\right\} $
is an interval (cf. \cite{MR1738952}). By the Facts \ref{fact3} and \ref{fact2} we have
$\mathcal{F}_{0}\not=\emptyset$ and $\mathcal{F}_{1}\not=\emptyset$.
Further, we clearly have on the one hand $0\leq\liminf_{n\rightarrow\infty}\frac{2\sum_{i=1}^{n}\log^{+}(a_{i}(x))}{d\left(B_{1}\cdots B_{n}\left(i\right),i\right)}$
and on the other hand, using (\ref{eq:distance-additive}) and (\ref{eq:lower_and_upper_bound_a_i>1}),
we estimate for all $x\in L_{r}\left(G\right)$,
\begin{equation}
\limsup_{n\rightarrow\infty}\frac{2\sum_{i=1}^{n}\log^{+}(a_{i}(x))}{d\left(B_{1}\cdots B_{n}\left(i\right),i\right)}\le\limsup_{n\rightarrow\infty}\frac{2\sum_{i=1}^{n}\log^{+}(a_{i}(x))}{\sum_{i=1}^{n}l(\xi_{i})}\le\lim_{n\rightarrow\infty}\frac{\sum_{i=1}^{n}l(\xi_{i})}{\sum_{i=1}^{n}l(\xi_{i})}=1.\label{eq:fact1 upper bound}
\end{equation}
This proves the fact.
\end{proof}

\section{Multifractal formalism for conformal graph directed Markov systems}
In this section  we  develop the  multifractal formalism for quotients of Birkhoff sums in the  framework of conformal graph directed Markov systems (cf.  \cite[Section 4.9]{MR2003772}).
Let us first briefly recall the definition a conformal graph directed Markov system  (\cite{MR2003772}). A conformal graph directed Markov system $\Phi$ is given by a finite set of vertices $V$, a family $(X_v)_{v\in V}$ of compact connected subsets of $\mathbb{R}^d$ with $X_v=\overline{\Int(X_v)}$ and an edge set $E\subset V\times V$ together with contractions $(\Phi_e)_{e\in E}$, where $\Phi_e: X_{t(e)}\rightarrow X_{i(e)}$. Here,  $t(e) \in V$ denotes the terminal vertex and  $i(e) \in V$ denotes the initial vertex of the edge $e$.  Moreover, $\Phi$ is endowed with an (edge) incidence matrix $A\in \{ 0,1\}^{E\times E }$ satisfying $A_{e,f}=1$ only if $X_{t(f)}=X_{i(e)}$.  We denote the associated Markov shift with alphabet $E$ and incidence matrix $A$ by $\Sigma:=\Sigma_{\Phi}$.

We always  assume that each $\Phi_e$ has a $C^1$-conformal extension satisfying a Hölder condition as stated in (4c, 4e) of  \cite[Section 4.2]{MR2003772}.  Moreover, we   assume that  $\Phi$ satisfies the open set condition and the cone condition as stated in (4b, 4d) \cite[Section 4.2]{MR2003772}.  Furthermore, we assume that the Markov shift $\Sigma_{\Phi}$  is finitely irreducible (\cite[page 5]{MR2003772}).

There is a natural coding map $p: \Sigma\rightarrow \bigcup_{v\in V}X_v$. The associated geometric potential $\zeta:=\zeta_{\Phi}:\Sigma\rightarrow \R^{-}$ is Hölder continuous with respect to the shift metric.
Let $\psi:\Sigma \rightarrow \R$ denote another Hölder continuous map. As in  (\cite{MR2719683}) we define the associated free energy function $t:\mathbb{R} \rightarrow \mathbb{R}\cup \{ +\infty\}$  which is for $\beta \in \R$ given by
$$ t(\beta) := \inf \left\{ t\in \R \mid \mathcal{P}(t\zeta +\beta \psi )\le 0 \right\}.$$ The function $t$ is a closed convex function with domain $dom(t)$, and we denote by $t^{-}$  (resp.  $t^{+}$)  its left (resp. right) derivative. We denote by $\hat{t}$ the Legendre transform of $t$. Define also
$$\alpha_{-}:=\inf \left\{ -t^{-}(x) \mid x\in dom(t) \right\} \text{ and } \alpha_{+}:=\sup \left\{ -t^{+}(x) \mid x\in dom(t) \right\}. $$

For $\alpha\in \R$ define the sets
$$\mathcal{F}_\alpha(\Phi,\psi) :=p\left\{ \omega \in \Sigma_{\Phi} \mid \lim_{k\rightarrow \infty} \frac{S_k \psi(\omega)}{S_k\zeta_{\Phi}(\omega)}=\alpha \right\}.$$

Set  $\alpha_{0}\coloneqq -t^{+}(0)$  and let
\[
\mathcal{F}_{\alpha}^{*}(\Phi,\psi)\coloneqq \begin{cases}
p\left\{ \omega \in \Sigma_{\Phi} \mid\limsup_{k\rightarrow \infty} \frac{S_k \psi(\omega)}{S_k\zeta_{\Phi}(\omega)}\geq\alpha\right\} , & \alpha\geq\alpha_{0},\\
p\left\{ \omega \in \Sigma_{\Phi} \mid \liminf_{k\rightarrow \infty} \frac{S_k \psi(\omega)}{S_k\zeta_{\Phi}(\omega)}\leq\alpha\right\} , & \alpha<\alpha_{0}.
\end{cases}
\]
Our main result is an extension of our multifractal formalism for conformal iterated function system (\cite{MR2719683}).
\begin{thm}\label{thm:mf} For every conformal graph directed Markov system $\Phi$ satisfying the above assumptions, for every Hölder continuous potential $\psi$ on the associated Markov shift and for every $\alpha \in (\alpha_{-},\alpha_{+})$ we have
$$ \dim_H(\mathcal{F}_\alpha(\Phi,\psi))=\dim_H(\mathcal{F}_{\alpha}^{*}(\Phi,\psi))=-\hat{t}(-\alpha).$$ For every $\alpha\in \mathbb{R}$  we have $\dim_H(\mathcal{F}_{\alpha}^{*}(\Phi,\psi))\le \max\{-\hat{t}(-\alpha),0\}$ and if $-\hat{t}(-\alpha)<0$ then $\mathcal{F}_{\alpha}^{*}(\Phi,\psi)=\emptyset$.
\end{thm}
Theorem \ref{thm:mf} can be proved by the same methods as in \cite{MR2719683}. The proof of the upper bound of the Hausdorff dimension follows from standard covering arguments (see for instance \cite[Proof of Theorem 4.2.13]{{MR2003772}}). To prove the lower bound of the Hausdorff dimension of the multifractal level sets, the key is to approximate the infinitely-generated conformal graph directed Markov system $\Phi$ by  finitely-generated subsystems. Previously, this method has been  used in \cite[Theorem 4.2.13]{{MR2003772}} to obtain the lower bound of the Hausdorff dimension of  the   limit set of a conformal graph directed Markov system. For the special case of conformal iterated function systems, the same  method  proved successful also for the level sets of multifractal decompositions of limit sets (\cite{MR2719683}). It is straightforward to extend the proof in \cite{MR2719683} to conformal graph directed Markov systems.  The key technical detail is the approximation property for the topological pressure for Hölder continuous potentials on finitely-irreducible topological Markov shifts \cite[Theorem 2.1.5]{MR2003772}.

\section{Proof of Theorem \ref{main} }

In this section we give the proof of  Theorem \ref{main}. For the interior points of the spectrum, the theorem is an application of our general multifractal formalism for conformal graph directed Markov systems. For the boundary points, additional arguments are required.

\subsection*{Multifractal formalism for the interior of the spectrum}
It is known that the radial limit set $L_{r}(G)$ of a free Fuchsian group
$G$ has a representation as an infinite conformal graph directed
Markov system $\Phi$ (see, for example, \cite{MR2337557}). The vertex set $V$ is given by the symmetric set of generators $G_0$. For each $g \in G_0$ the compact connected set $X_g\subset \R$ is given by the closure of ${b_s}$, where $b_s$   is the projection of the face $s\in \mathcal{F}$ such that  $g=g_s$ is the face-pairing transformation with respect to  Dirichlet fundamental domain of $G$ (see (\ref{bs-def}) in Section \ref{subsec:The-canonical-Markov} for the details).  Note that by conjugating the group $G$, we may assume that the point at infinity does not belong to the limit set of $G$.
The edge set $E$ is given by
 \begin{align*}
E:=I & \coloneqq\left\{ g_{1}\gamma^{n}g_{2}\mid\gamma\in\Gamma_{0},g_{1},g_{2}\in G_{0}\setminus\{\gamma^{\pm1}\},n\in\N\right\} \cup\\
 & \quad\quad\cup\left\{ g_{1}hg_{2}\mid g_{1}\in G_{0},h\in H_{0}\setminus\{g_{1}^{-1}\},g_{2}\in G_{0}\setminus\{h^{-1}\}\right\} .
\end{align*}
For $\eta_1\dots\eta_n\in E$ we define  $t(\eta_1\dots\eta_n):=\eta_1$ and $i(\eta_1\dots\eta_n):=\eta_{n-2}$.  We consider the incidence matrix $A\in \{0,1\}^{E\times E}$ which satisfies for  $e=\eta_1\dots\eta_n$ and $f=\tau_1\dots\tau_m\in E$ that  $A(e,f)=1$ if and only if $\eta_{n-1}=\tau_1$ and $\eta_{n}=\tau_2$. For such edges we define
\[ \Phi_{e,f}:X_{t(f)}\rightarrow X_{i(e)},\quad \Phi_{e,f}:=\eta_1^{-1}\dots \eta_{n-2}^{-1}.
\]
 Note that the Markov shift $\Sigma_{\Phi}$ associated with the graph directed Markov system $\Phi$ coincides with the Markov shift $\Sigma_A$ in Section \ref{subsec:Inducing-and-the-Gibbs_Markov_Map}.  Also note that the coding map $p$ of $\Phi$ is the inverse of the coding map $\pi$ defined in Section \ref{subsec:Inducing-and-the-Gibbs_Markov_Map}.  Observe that $\zeta_{\Phi}$ is equal to the potential $\phi$ defined in Section \ref{subsec:Inducing-and-the-Gibbs_Markov_Map}.  Define  $\psi:\Sigma_{\Phi} \rightarrow R$ as in Section \ref{subsec:Inducing-and-the-Gibbs_Markov_Map} and recall from   Lemma \ref{lem:pressure is pressure} and  Proposition \ref{pro:exTf(beta)} that   $t$ associated with $\Phi$ and $\psi$  coincides with the  real-analytic function $t$ defined in Section \ref{subsec:Inducing-and-the-Gibbs_Markov_Map}.  Hence, by Theorem \ref{thm:mf},  we have for $\alpha\in (\alpha_-,\alpha_+)$,
\begin{equation}
\dim_{H}\left(\mathcal{F}_\alpha(\Phi,\psi)\right)=\dim_{H}\left(\mathcal{F}^{*}_\alpha(\Phi,\psi)\right)=-\hat{t}(-\alpha).\label{eq:ExactSpec=00003DleqSpec}
\end{equation}

By the definition of $\Phi$ and the potential $\psi$, and by Lemma \ref{lem:pressure is pressure} we have
\[
\mathcal{F}_\alpha(\Phi,\psi)=\mathcal{F}_\alpha \cap \mathcal{D} \text{ and }
\mathcal{F}^{*}_\alpha(\Phi,\psi)=\mathcal{F}^{*}_\alpha \cap \mathcal{D}.
\]
Combining this with  Fact \ref{fact1},  we see that  $\alpha_-=0$ and $\alpha_+=1$.
Finally, since  $\mathcal{F}_{\alpha}$ (resp. $\mathcal{F}^{*}_\alpha$) is a countable union of bi-Lipschitz images of $\mathcal{F}_{\alpha}\cap \mathcal{D}$ (resp. $\mathcal{F}^{*}_\alpha\cap \mathcal{D})$,  we conclude that $\dim_H({\mathcal{F}_\alpha})=\dim_H({\mathcal{F}^{*}_\alpha})=-\hat{t}(-\alpha)$.



\subsection*{Asymptotic behaviour for the left boundary point}

By Fact \ref{fact2} we have
\[
\dim_{\mathrm{H}}\left(\mathcal{F}_{0}\right)\geq\dim_{\mathrm{H}}\left(L_{c}\right)=\delta_{c}.
\]
Combining (\ref{eq:ExactSpec=00003DleqSpec}) and Lemma \ref{lem:Asymp_t_right}
the lower bound follows from
\[
\dim_{\mathrm{H}}\left(\mathcal{F}_{0}\right)\leq\lim_{\alpha\to0}\dim_{\mathrm{H}}\left(\mathcal{F}_{\alpha}^{*}\right)=\delta_{c}.
\]

\subsection*{Asymptotic behaviour for the right boundary point}

Combining Fact \ref{fact3} and \ref{fact4} we obtain $\dim_{\mathrm{H}}\left(\mathcal{F}_{1}\right)\geq\dim_{\mathrm{H}}\left(\mathcal{J}\right)=1/2$.

Now this observation and the fact that $\mathcal{F}_{1}\subset\mathcal{F}_{\alpha}^{*}$
for every $\alpha_{0}<\alpha<1$ together with Lemma \ref{lem:tAsymptoticFor-right}
finally gives

\[
\dim_{\mathrm{H}}\mathcal{F}_{1}\leq\lim_{\alpha\nearrow1}\dim_{\mathrm{H}}\mathcal{F}_{\alpha}^{*}\leq1/2.
\]
This proves the claim for the properties of the right boundary point
of the spectrum.

\subsection*{The derivatives in the boundary points}

For the derivative of $f$ at the boundary points we use the general
identity for Legendre transforms
\[
f'\left(\alpha\right)=\left(\widehat{t}\right)'\left(-\alpha\right)=-\left(t'\right)^{-1}\left(-\alpha\right),\,\,\,\,\alpha\in\left(0,1\right).
\]
The real-analyticity then gives
\[
\lim_{\alpha\searrow0}f'\left(\alpha\right)=-\lim_{\alpha\nearrow1}f'\left(\alpha\right)=+\infty.
\]
\subsection*{Irregular set}
The fact that the set of points for which the  mean cusp winding number does not exist has full Hausdorff dimension $\delta(G)$ follows from \cite{MR1759398} by exhaustion of $\Sigma_{A}$  with finite alphabet subsystems.

\providecommand{\bysame}{\leavevmode\hbox to3em{\hrulefill}\thinspace}
\providecommand{\MR}{\relax\ifhmode\unskip\space\fi MR }
\providecommand{\MRhref}[2]{%
  \href{http://www.ams.org/mathscinet-getitem?mr=#1}{#2}
}
\providecommand{\href}[2]{#2}

\end{document}